\documentclass[a4paper,12pt,fleqn]{amsart}
\usepackage{amsmath}
\usepackage{amsthm}
\usepackage{amssymb}
\usepackage[foot]{amsaddr}
\usepackage[inline]{enumitem}
\usepackage{color}
\usepackage{xfrac}
\usepackage[normalem]{ulem}
\usepackage[all]{xy}
\usepackage[dvipsnames]{xcolor}
\usepackage[utf8]{inputenc}
\usepackage{setspace}
\usepackage{currfile}
\usepackage{comment}
\usepackage{mathabx}
\usepackage{textcomp}
\usepackage{cancel}
\usepackage{mathrsfs}
\usepackage[unicode]{hyperref}
\usepackage{cleveref}
\usepackage{thmtools}
\usepackage{booktabs}

\usepackage[english]{babel}
\addto\extrasenglish{
  
}

\usepackage[ 
sorting=nyt,
style=numeric-comp,
maxbibnames=4,
giveninits=true,
block=space,
backend=biber
]{biblatex}
\usepackage{newunicodechar}

\usepackage[unicode]{hyperref}

\renewbibmacro{in:}{%
  \ifentrytype{article}{}{\printtext{\bibstring{in}\intitlepunct}}}

\DeclareSourcemap{
  \maps[datatype=bibtex, overwrite]{
    \map{
      \step[fieldset=language, null]
      \step[fieldset=shorthand, null]
    }
  }
}

\DeclareFieldFormat[unpublished]{note}{#1\addcomma\space}
\DeclareFieldFormat[thesis]{type}{#1\addcomma\space}

\DeclareBibliographyDriver{unpublished}{%
  \usebibmacro{bibindex}%
  \usebibmacro{begentry}%
  \usebibmacro{author}%
  \setunit{\labelnamepunct}\newblock
  \usebibmacro{title}%
  \newunit
  \printlist{language}%
  \newunit\newblock
  \usebibmacro{byauthor}%
  \newunit\newblock
  \printfield{howpublished}%
  \newunit\newblock
  \printfield{note}%
  \newunit\newblock
  \usebibmacro{location+date}%
  \newunit\newblock
  \iftoggle{bbx:url}
  {\usebibmacro{doi+eprint+url}}
  {}%
  \newunit\newblock
  \usebibmacro{addendum+pubstate}%
  \setunit{\bibpagerefpunct}\newblock
  \usebibmacro{pageref}%
  \newunit\newblock
  \iftoggle{bbx:related}
  {\usebibmacro{related:init}%
    \usebibmacro{related}}
  {}%
  \usebibmacro{finentry}}

\hypersetup{
  colorlinks,
  linkcolor={blue!80!black},
  citecolor={blue!80!black},
  urlcolor={blue!80!black}
}

\DeclareFontFamily{U}{mathx}{\hyphenchar\font45}
\DeclareFontShape{U}{mathx}{m}{n}{
  <5> <6> <7> <8> <9> <10>
  <10.95> <12> <14.4> <17.28> <20.74> <24.88>
  mathx10
}{}
\DeclareSymbolFont{mathx}{U}{mathx}{m}{n}
\DeclareFontSubstitution{U}{mathx}{m}{n}
\DeclareMathAccent{\widecheck}{0}{mathx}{"71}
\DeclareMathAccent{\wideparen}{0}{mathx}{"75}

\newcommand{\coleq}{\mathrel{\mathop:}=}

\newcommand{\e}{\varepsilon}

\declaretheorem[name=Theorem,refname={theorem,theorems},Refname={Theorem,Theorems}]{theorem}

\declaretheorem[name=Proposition,refname={proposition,propositions},Refname={Proposition,Propositions},sibling=theorem]{proposition}

\declaretheorem[name=Remark,refname={remark,remarks},Refname={Remark,Remarks},sibling=theorem]{remark}

\crefname{proposition}{proposition}{propositions}

\bibliography{wong}

\newcommand{\cB}{{\mathscr{B}}}
\newcommand{\cC}{\mathscr{C}}
\newcommand{\cD}{\mathscr{D}}
\newcommand{\cE}{\mathscr{E}}
\newcommand{\cF}{\mathscr{F}}

\newcommand{\cU}{\mathscr{U}}

\newcommand{\cO}{\mathscr{O}}
\newcommand{\cS}{\mathscr{S}}
\newcommand{\N}{\mathbb{N}}
\newcommand{\R}{\mathbb{R}}
\newcommand{\C}{\mathbb{C}}
\newcommand{\pd}{\partial}

\newcommand{\abso}[1]{\left|#1\right|}
\newcommand{\norm}[1]{\left\lVert#1\right\rVert}

\begin{document}
\title{Quasinormable Fr\'echet spaces and M.~W.~Wong's inequality}
\author{Eduard A.\ Nigsch$^1$}
\address{$^1$Institute of Analysis and Scientific Computing, TU Wien, Wiedner Hauptstraße 8-10, 1040 Wien, Austria}
\email[Corresponding author]{eduard.nigsch@tuwien.ac.at}
\author{Norbert Ortner$^2$}
\address{$^2$Institut für Mathematik, Universität Innsbruck, Technikerstraße 13, 6020 Innsbruck, Austria}
\email{mathematik1@uibk.ac.at}
\date{January 2024}
\begin{abstract}
A short proof of M.~W.~Wong's inequality $\norm{J_{-s}\varphi}_p \le \e \norm{J_{-t}\varphi}_p + C \norm{\varphi}_p$ is given.
\end{abstract}

\maketitle



\section{Introduction and Notation.}\label{sec1}

In \cite[Prop. 2.2, p.~1696 and Prop.~4.3 (b), p.~1701]{BNO} we have proved that the space $\cD'_{L^q}$, $1 \le q \le \infty$, is a compactly regular (LB)-space. The first proof uses M.~W.~Wong's inequality \cite[Theorem 2.1, p.~101]{Wo}.

However, the compact regularity of $\cD'_{L^q}$, $1 \le q \le \infty$, also is an easy consequence of the quasinormability of $\cD_{L^p}$ and $\dot\cB$ ($1 \le p \le \infty$, $1/p + 1/q = 1$). Therefore, we recall a characterization of the quasinormability of Fr\'echet spaces in \Autoref{prop1}. The definition of the spaces $\cD_{L^p}$ and $\dot\cB$ and an equivalent description by means of the Bessel kernels is recalled in \Autoref{sec3}. The quasinormability of $\cD_{L^p}$ and $\dot\cB$ is stated in \Autoref{prop2} and the consequences for the strong duals $\cD'_{L^q}$, $1 \le q \le \infty$, are quoted in \Autoref{prop3}. In \Autoref{sec4}, M.~W.~Wong's inequality is derived from the quasinormability of the spaces $\cD_{L^p}$, $\dot \cB$. We employ the notation from \cite{S}.

\section{Characterization of quasinormable Fr\'echet spaces.}\label{sec2}

\begin{proposition}\label{prop1}
Let $E$ be a Fr\'echet space, $\{\cU_k : k \in \N \}$ a decreasing basis of absolutely convex closed neighborhoods of $0$, $E_k' \coleq (E')_{\cU_k^\circ} = \langle \cU_k^\circ \rangle$ the local Banach spaces and $q_k$ the Minkowski functional of $\cU_k^\circ$ (a seminorm in $E_k'$). Then:
\begin{enumerate}[label=(\roman*)]
\item\label{prop1i} The quasinormability of the Fr\'echet space $E$ is equivalent to the inclusion relation
\[ \forall k\ \exists l>k\ \forall m>l\ \forall \e>0\ \exists C >0: \cU_l \subset \e \cU_k + C \cU_m. \]
\item\label{prop1ii} If $E$ is quasinormable we obtain
\[ E_b' = \varinjlim_{k} E_k',\quad E_k' \hookrightarrow E_l' \hookrightarrow E_m' \quad \textrm{if }k<l<m. \]
Furthermore, the following inequality holds:
\[ \forall k\ \exists l>k\ \forall m>l\ \forall \e>0\ \exists C>0: q_l(e') \le \e q_k(e') + C q_m(e')\quad \forall e' \in E_k'. \]
\end{enumerate}
The inductive limit $\varinjlim_k E_k'$ is compactly regular, ultrabornological and complete.
\end{proposition}
\begin{proof}
\ref{prop1i} follows from \cite[Theorem 30, p.~183]{BB}.

\ref{prop1ii} $E_b' = \varinjlim_k E_k'$ is a consequence of \cite[p.~84]{B}. The inequality follows from the inclusion relation in \ref{prop1i}. The properties of the inductive limit $\varinjlim_k E_k'$ are consequences of \cite[Theorem 6.4, p.~112 and Corollary 6.5, p.~113]{We}.
\end{proof}

\section{
\texorpdfstring{%
Definition of the spaces $\cD_{L^p}$ and $\dot \cB$ by means of Bessel Kernels. Quasinormability of $\cD_{L^p}$ and $\dot\cB$.
}{%
Definition of the spaces DLp and Ḃ by means of Bessel Kernels. Quasinormability of DLp and Ḃ.
}}
\label{sec3}

The \emph{Bessel kernels} $L_z$, $z \in \C$, are defined by $\cF L_z = (1+\abso{x}^2)^{-z/2} \in \cO_M$, i.e., $L_z \in \cO_C'$ (see \cite[(II, 3; 20), p.~47 and Ex.~2° in Chap.~VII, §8, p.~271]{S}). Furthermore, we have
\[ L_z * L_w = L_{z+w}\quad \textrm{for}\ (z,w) \in \C^2 \]
(\cite[(VI, 8; 5), p.~204]{S}). 
The \emph{Fr\'echet spaces $\cD_{L^p}$ ($1 \le p \le \infty$) and $\dot\cB$} are defined by
\begin{align*}
\cD_{L^p} &= \{\,\varphi \in \cE(\R^n)\ |\ \forall \alpha \in \N_0^n: \pd^\alpha \varphi \in L^p(\R^n)\,\}, \\
\dot\cB &= \{\,\varphi \in \cE(\R^n)\ |\ \forall \alpha \in \N_0^n: \pd^\alpha \in \cC_0(\R^n)\,\}
\end{align*}
(\cite[p.~199]{S}). We infer
\[ \cD_{L^p} = \varprojlim_{k} (L_k * L^p) \quad \textrm{for }1 < p < \infty \]
in virtue of \cite[Theorem 7, p.~36]{C}, \cite[Corollary 2, p.~347]{M} or \cite[Theorem 3, p.~135]{St}.

For $p=1$ or $p=\infty$ we have
\[ \varprojlim_k (L_k * L_p) \supset \cD_{L^p} \quad\textrm{and}\quad \varprojlim_k (L_k * \cC_0) \supset \dot\cB. \]
The reverse inclusions follow from the relations
\[ \pd^\alpha \delta = (L_s * \pd^\alpha \delta) * L_{-s}\quad\textrm{and}\quad L_s * \pd^\alpha\delta \in L^1\textrm{ if }s > \abso{\alpha}. \]
$L_z * L_p$ and $L_z * \cC_0$ are Banach spaces, instead of $\cD_{L^p}$ Calder\'on writes $L^p_\infty$ \cite[p.~39]{C}.

The \emph{Bessel potentials} $L_z * \varphi$, $\varphi \in \cS(\R^n)$, $z \in \C$, are identical with the pseudo-differential operators $J_z \varphi$ considered in \cite[p.~98]{Wo}. Due to $L_z \in \cO_C'$ the Bessel potentials $L_z * S$ are defined for all temperate distributions $S \in \cS'$.

Furthermore we remark that the topologies of $\cD_{L^p}$ and $\dot \cB$ are generated by the seminorms
\[ \varphi \mapsto \norm{L_{-k} * \varphi}_p,\quad \norm{L_{-k} * \varphi}_\infty,\quad k \in \N_0, \]
or, equivalently, by
\[ \varphi \mapsto \norm{L_{-s} * \varphi}_p,\quad \norm{L_{-s} * \varphi}_\infty,\quad s>0. \]

\begin{proposition}\label{prop2}
The spaces $\cD_{L_p}$, $1 \le p \le \infty$, and $\dot\cB$ are quasinormable.
\end{proposition}

\begin{proof}

\underline{1st proof.} By \cite[Theorem 1, p.~766]{Va} and \cite[3.2 Theorem, p.~414]{Vo} we obtain the isomorphisms
\[ \cD_{L^p} \cong \ell^p \widehat\otimes s,\ 1 \le p \le \infty,\ \textrm{and}\ \dot\cB \cong c_0 \widehat\otimes s \]
where $s$ is the space of rapidly decreasing sequences. Then the permanence property \cite[Chapter II, Proposition 13, p.~76]{G} and the nuclearity of the space $s$ imply the quasinormability of $\cD_{L^p}$ and $\dot\cB$.

\underline{2nd proof.} We take the sets $L_k * B_{1,p}$, $1 \le p \le \infty$, $k \in \N_0$, as a decreasing basis of neighborhoods in $\cD_{L^p}$, $B_{1,p}$ the unit ball of $L^p$. By \Autoref{prop2}, the quasinormability of $\cD_{L^p}$ is equivalent to the following inclusion relation:
\begin{multline*}
\forall k \in \N_0\ \exists l>k\ \forall m > l\ \forall \e>0\ \exists C>0:\\
L_l * B_{1,p} \subset \e L_k * B_{1,p} + C L_m * B_{1,p}.
\end{multline*}
By the group law for the Bessel kernels we obtain (with $\varphi \in \cD \setminus \{0\}$, $\varphi_R (x) \coleq R^{-n} \varphi(Rx)$):
\[ L_l * B_{1,p} = (\delta - \varphi_R)*L_{l-k} * L_k * B_{1,p} + \varphi_R * L_{l-m} * L_m * B_{1,p}. \]
Let $f \in L_l * B_{1,p}$ and $f = f_1 + f_2$ with
\begin{align*}
f_1 &\coleq ((\delta-\varphi_R) * L_{l-k}) * L_k * g_1,\quad g_1 \in B_{1,p}\ \textrm{and}\\
f_2 & \coleq (L_{l-m} * \varphi_R) * L_m * g_2,\quad g_2 \in B_{1,p}.
\end{align*}
Then,
\[ \norm{L_{-k} * f_1}_p \le \norm{(\delta - \varphi_R) * L_{l-k}}_1 \underbrace{\norm{g_1}_p}_{\le 1} \]
and
\[ \norm{L_{-m} * f_2}_p \le \norm{L_{l-m} * \varphi_R}_1 \underbrace{\norm{g_2}_p}_{\le 1}, \]
i.e., $f_1 \in \norm{(\delta-\varphi_R)*L_{l-k}}_1 (L_k * B_{1,p})$ and $f_2 \in \norm{L_{l-m} * \varphi_R}_1(L_m * B_{1,p})$ wherein we have used the continuity of the convolution $* : L^1 \times L^p \to L^p$.

Hence,
\[ f = f_1 + f_2 \in \e(L_k * B_{1,p}) + C(L_m * B_{1,p}) \]
if we set $\e \coleq \norm{(\delta - \varphi_R)*L_{l-k}}_1$ and $C \coleq \norm{L_{l-m} * \varphi_R}_1$. Thus, the inclusion in \Autoref{prop1} \ref{prop1i},
\[ L_l * B_{1,p} \subset \e L_k * B_{1,p} + C L_m * B_{1,p}, \]
is shown. The proof for $\dot \cB$ follows if $p=\infty$.
\end{proof}

\begin{proposition}\label{prop3}
Let $1/p + 1/q = 1$, $1 \le q \le \infty$. Thje strong duals $\cD'_{L^q}$ of $\cD_{L^p}$, $1 \le p < \infty$ and of $\dot\cB$ are compactly regular, complete, ultrabornological (LB)-spaces.
\end{proposition}

The proof is a consequence of \Autoref{prop1} \ref{prop1ii}.

\section{M.~W.~Wong's inequality.}\label{sec4}

\begin{proposition}[{\cite[Theorem 2.1, p.~101]{Wo}}]\label{prop4}
If $1 \le q < \infty$, $0 < s < t$ then $\forall \e>0$ $\exists C>0$:
\[ \norm{L_{-s} * \varphi}_q \le \e \norm{L_{-t} * \varphi}_q + C \norm{\varphi}_q\quad \forall \varphi \in \cS(\R^n). \]
\end{proposition}
\begin{proof}
By taking into account that
\[ \norm{f}_q = \sup_{g \in B_{1,p}} \abso{\langle g, f \rangle},\quad 1/p + 1/q = 1, \]
Wong's inequality is equivalent to the inclusion relation
\[ L_{-s} * B_{1,p} \subset \e ( L_{-t} * B_{1,p}) + C B_{1,p} \]
which is equivalent to \ref{prop1i} in \Autoref{prop1} if we set $-s = l-m$, $m>l$, $-t = k-m$, $l>k$ and if we take into account the equivalence of the seminorms (in $\cD_{L^p}$)
\begin{gather*}
\varphi \mapsto \norm{L_{-s} * \varphi}_p,\quad \varphi \in \cS,\ s > 0\ \textrm{and}\\
\varphi \mapsto \norm{L_{-k} * \varphi}_p,\quad \varphi \in \cS,\ k \in \N_0,
\end{gather*}
and, furthermore, the quasinormability of the Fr\'echet space $\cD_{L^p}$.
\end{proof}

\begin{remark}In virtue of the quasinormability of the space $\cD_{L^1}$ the inequality in \Autoref{prop4} also is valid if $q=\infty$.
\end{remark}

\newcommand{\bibarxiv}[1]{arXiv: \href{http://arxiv.org/abs/#1}{\texttt{#1}}}
\newcommand{\bibdoi}[1]{{\sc doi:} \href{http://dx.doi.org/#1}{\texttt{#1}}}

\printbibliography

@incollection {B,
    AUTHOR = {Bierstedt, Klaus D.},
     TITLE = {An introduction to locally convex inductive limits},
 BOOKTITLE = {Functional analysis and its applications ({N}ice, 1986)},
    SERIES = {ICPAM Lecture Notes},
     PAGES = {35--133},
 PUBLISHER = {World Sci. Publishing},
 address = {Singapore},
      YEAR = {1988},
       DOI = {10.1007/s13116-009-0018-2},
}

@Article{BB,
 Author = {Bierstedt, Klaus D. and Bonet, Jos{\'e}},
 Title = {Some aspects of the modern theory of {Fr{\'e}chet} spaces},
 Journal = {RACSAM, Rev. R. Acad. Cienc. Exactas F{\'i}s. Nat., Ser. A Mat.},
 Volume = {97},
 Number = {2},
 Pages = {159--188},
 Year = {2003},
}

@article{BNO,
        author = {Bargetz, C. and Nigsch, Eduard A. and Ortner, N.},
        title = "A simpler description of the $\kappa$-topologies on the spaces $\mathcal{D}_{L^p}$, $L^p$, $\mathcal{M}^1$",
        eprinttype={arxiv},
        journal="Math. Nachr.",
        eprint = {1711.06577},
        pages = {1691--1706},
        year = 2020,
        volume = 293,
        number = 9,
        doi="10.1002/mana.201900109"
}

@article{C,
 Author = {Calder{\'o}n, Alberto P.},
 Title = {Lebesgue spaces of differentiable functions and distributions},
 Year = {1961},
 journal = "Proc. {Sympos}. {Pure} {Math}.",
 volume = 4,
 pages = {33--49},
}

@Article{G,
    Author = {Alexandre Grothendieck},
    Title = "Produits tensoriels topologiques et espaces nucl{\'e}aires",
    Journal = {Mem. Am. Math. Soc.},
    Volume = {16},
    Year = {1955},
}

@Book{M,
 Author = {Maz'ya, Vladimir G.},
 Title = {Sobolev spaces},
 Year = {1985},
 address = {Berlin},
 publisher = {Springer}
}

@book {S,
        author = {Schwartz, Laurent},
        title = {Th{\'e}orie des distributions},
        edition = "Nouvelle \'edition, enti{\`e}rement corrig\'ee, refondue et augment\'ee",
        publisher = {Hermann},
        address = {Paris},
        year = {1966}   
}

@book{St,
    author = {Elias M. Stein},
    title = {Singular integrals and differentiability properties of functions},
    publisher = {Princeton University Press},
    year = 1970,
    address = {Princeton, N.J.}
}

@incollection{Va,
 Editor = {Nachbin, Leopoldo},
 bookTitle = {Mathematical analysis and applications. {Part} {B}. {Essays} dedicated to {Laurent} {Schwartz} on the occasion of his 65th birthday},
 author={M. Valdivia},
 title={On the space {$\mathcal{D}_{L^p}$}},
 Series = {Adv. Math., Suppl. Stud.},
 Volume = {7b},
 Year = {1981},
 Publisher = {New York, NY: Academic Press},
}

@incollection {Vo,
    AUTHOR = {Vogt, Dietmar},
     TITLE = {Sequence space representations of spaces of test functions and
              distributions},
 BOOKTITLE = {Functional analysis, holomorphy, and approximation theory
              ({R}io de {J}aneiro, 1979)},
    SERIES = {Lecture Notes in Pure and Appl. Math.},
    VOLUME = {83},
     PAGES = {405--443},
 PUBLISHER = {Dekker},
   ADDRESS = {New York},
      YEAR = {1983},
    EDITOR = {Zapata, Guido I.}
}

@Book{We,
    AUTHOR = {Jochen Wengenroth},
     TITLE = {Derived functors in functional analysis},
      YEAR = {2003},
 PUBLISHER = {Springer},
   ADDRESS = {Berlin}
}

@Article{Wo,
 Author = {Wong, M. W.},
 Title = {Erhling's inequality and pseudo-differential operators on {{\(L^p(\mathbb R^n)\)}}},
 FJournal = {Cubo},
 Journal = {Cubo},
 Volume = {8},
 Number = {1},
 Pages = {97--108},
 Year = {2006},
 Language = {English},
 zbMATH = {5056658},
 Zbl = {1104.35082}
}

\end{document}